\documentclass[11pt]{article}
\pdfoutput=1

\usepackage[letterpaper,margin=1in]{geometry}
\usepackage{microtype}
\usepackage{mathtools}
\usepackage{booktabs}
\usepackage{amsmath,amssymb,amsthm,amsfonts}
\usepackage[pdftex,colorlinks,citecolor=blue,linkcolor=blue,urlcolor=black,bookmarks=false]{hyperref}
\usepackage[nameinlink]{cleveref}
\crefname{ineq}{inequality}{inequalities}
\creflabelformat{ineq}{#2{\upshape(#1)}#3}
\usepackage{doi}

\numberwithin{table}{section}

\newcommand\arXiv[1]{arXiv:\href{http://arXiv.org/abs/#1}{#1}}

\newtheorem{theorem}{Theorem}[section]
\newtheorem{lemma}[theorem]{Lemma}

\newtheorem{proposition}[theorem]{Proposition}
\newtheorem{corollary}[theorem]{Corollary}

\newtheorem{question}[theorem]{Question}

\theoremstyle{definition}
\newtheorem{definition}[theorem]{Definition}
\newtheorem{remark}[theorem]{Remark}

\newcommand\C{\mathbb{C}}
\newcommand\Quat{\mathbb{H}}
\newcommand\F{\mathbb{F}}
\newcommand\R{\mathbb{R}}

\newcommand\N{\mathbb{N}}
\newcommand\SL{\mathop{\mbox{\textup{SL}}}\nolimits}
\newcommand\Sp{\mathop{\mbox{\textup{Sp}}}\nolimits}
\newcommand\SU{\mathop{\mbox{\textup{SU}}}\nolimits}
\newcommand\GL{\mathop{\mbox{\textup{GL}}}\nolimits}
\newcommand\SO{\mathop{\mbox{\textup{SO}}}\nolimits}
\renewcommand\O{\mathop{\mbox{\textup{O}}}\nolimits}
\newcommand\U{\mathop{\mbox{\textup{U}}}\nolimits}
\newcommand\irr{{\mathop\textup{Irr}}}
\newcommand\Tr{{\mathop\textup{Tr}}}
\newcommand\diag{{\mathop\textup{diag}}}

\title{Matrix multiplication via matrix groups}
\author{Jonah Blasiak\thanks{Department of Mathematics, Drexel University, Philadelphia, PA 19104,  USA,  jblasiak@gmail.com} \and Henry Cohn\thanks{Microsoft Research New England, One Memorial Drive, Cambridge, MA 02142,  USA, cohn@microsoft.com} \and Joshua A.\ Grochow\thanks{Departments of Computer Science and Mathematics, University of Colorado Boulder, Boulder, CO 80309, USA, jgrochow@colorado.edu} \and Kevin Pratt\thanks{School of Computer Science, Carnegie Mellon University, Pittsburgh, PA 15213, USA, kpratt@andrew.cmu.edu} \and Chris Umans\thanks{Department of Computing and Mathematical Sciences, California Institute of Technology, Pasadena, CA 91125, USA, umans@cs.caltech.edu}}
\date{}

\begin{document}

\maketitle 

\thispagestyle{empty}

\begin{abstract}
In 2003, Cohn and Umans proposed a group-theoretic approach to bounding the exponent of matrix multiplication. Previous work within this approach ruled out certain families of groups as a route to obtaining $\omega = 2$, while other families of groups remain potentially viable. In this paper we turn our attention to matrix groups, whose usefulness within this framework was relatively unexplored.

We first show that groups of Lie type cannot prove $\omega=2$ within the group-theoretic approach. This is based on a representation-theoretic argument that identifies the second-smallest dimension of an irreducible representation of a group as a key parameter that determines its viability in this framework. Our proof builds on Gowers' result concerning product-free sets in quasirandom groups. We then give another barrier that rules out certain natural matrix group constructions that make use of subgroups that are far from being self-normalizing.

Our barrier results leave open several natural paths to obtain $\omega = 2$ via matrix groups. To explore these routes we propose working in the continuous setting of Lie groups, in which we develop an analogous theory. Obtaining the analogue of $\omega=2$ in this potentially easier setting is a key challenge that represents an intermediate goal short of actually proving $\omega = 2$. We give two constructions in the continuous setting, each of which evades one of our two barriers.
\end{abstract}

\newpage
\setcounter{page}{1}

\section{Introduction}\label{sec:intro}
The exponent of matrix multiplication is the smallest number $\omega$ such that for each $\varepsilon > 0$, there exists an algorithm for multiplying two $n \times n$ matrices using $O(n^{\omega+\varepsilon})$ field operations.  It is clear that $\omega \ge 2$, and a long line of work has led to the best upper bound currently known of $\omega < 2.37286$ from \cite{alman2021refined}. It is a longstanding and well-known open problem to resolve the conjecture that $\omega = 2$.

In \cite{cu2003} a group-theoretic approach to bounding $\omega$ was proposed. Given any finite group $G$ and three subsets of $G$ satisfying a certain condition (the \emph{triple product property}), this approach yields an upper bound on $\omega$ by reducing an instance of matrix multiplication to multiplication in the group algebra of $G$. This approach can capture the Coppersmith--Winograd family of algorithms \cite{cohn2005group}, which includes the current record bound. While some barriers are known for the group-theoretic approach, for instance that abelian groups cannot not prove $\omega < 3$ (see \cite[Lemma 3.1]{cu2003}) or that certain generalizations of the constructions in \cite{cohn2005group} cannot achieve $\omega = 2$  \cite{blasiak2017cap,BCCGU,sawin2018bounds}, the possibility that one could show $\omega = 2$ using a suitable family of nonabelian groups remains wide open.\footnote{Other barrier results rule out different generalizations of the Coppersmith--Winograd approach, such as \cite{AFLG,AWBarrier, AlmanBarrier, CVZ}, but it remains unclear the extent of the implications those barriers have for the group-theoretic approach.}

Previous work on the group-theoretic approach has focused mainly on families of very \emph{non-simple} groups, i.e., groups built up from simple groups through repeated group extension. For example, the best bounds known on $\omega$ can be obtained using a semidirect product of the symmetric groups with direct products of abelian groups. At the same time, the aforementioned barriers rule out these kinds of constructions and certain generalizations  \cite{blasiak2017cap,BCCGU,sawin2018bounds}. But this has left the simple groups---in some sense the opposite end of the spectrum of finite groups---largely unexplored.

In this paper we address this gap in knowledge by studying \emph{finite groups of Lie type}\footnote{Specifically, by a group of Lie type we mean any one of the (possibly twisted) Chevalley groups, including the Suzuki and Ree groups, or the quotient of such a group by its center. A Chevalley group is the fixed points of a Steinberg endomorphism in a semisimple algebraic group over a finite field (see Definition~21.6 and Table~22.1 in \cite{MalleTesterman}).
Among others, this list includes $\SL(n,q)$, $\SU(n,q)$, $\SO(2n+1,q)$, $\Sp(2n,q)$, $\SO^+(2n,q)$, and $\SO^-(2n,q)$. Obtaining simple groups can require taking the quotient by the center, but that does not change our conclusions, such as \Cref{cor:liebarrier}.} in the framework of \cite{cu2003}. This is an important class of groups that contains all of the finite simple groups except alternating or cyclic groups and finitely many sporadic groups. Some good examples to keep in mind are the classical matrix groups such as the group $\SL(n,q)$ of determinant $1$ matrices over the finite field $\F_q$ or the group of $n \times n$ orthogonal matrices over $\F_q$ with respect to a quadratic form.

\subsection{Results}
We start by showing that triple product property constructions (see \Cref{def:tpp}) in groups of Lie type cannot prove any bound on $\omega$ better than $2+\varepsilon$ for some absolute constant $\varepsilon > 0$ (\Cref{cor:liebarrier}). This resolves a question asked in \cite{cu2003}. Our proof combines a representation-theoretic argument with known bounds on the dimension and number of irreducible representations in groups of Lie type \cite{landazuri1974minimal,fulman2012bounds}. 
More broadly, we identify the second-smallest dimension of an irreducible representation as a key parameter of a group that determines its viability for the approach of \cite{cu2003}: \Cref{thm:gowerstrick} shows that groups where this quantity is large cannot yield good bounds on $\omega$. For example, any family of groups for which the second-smallest dimension of an irreducible representation grows as a power of the size of the group cannot yield $\omega = 2$. It had been known since \cite{cu2003} that the largest dimension played a key role in the quality of the bound, but small dimensions were not previously understood to be relevant.

This first barrier builds on Gowers' theorem on product-free sets in quasirandom groups \cite{gowers2008quasirandom}. We note that whereas Gowers' result involves the minimum dimension of a \emph{nontrivial} representation, the additional structure of our problem allows us to consider the second-smallest dimension of an irreducible representation (in other words, we can skip any other representations of dimension~$1$). This gives us lower bounds in groups where Gowers' result does not apply, such as $\SL(n,q)$. It is interesting that while triple product property constructions in abelian groups cannot yield nontrivial bounds on $\omega$, this barrier shows that \emph{highly nonabelian} groups also have significant limitations.

Next we show in \Cref{thm:normbarrier} that subgroups with large normalizers cannot be used in a triple product property construction to obtain $\omega = 2$. This barrier is particularly effective in the setting of matrix groups. For example, one cannot obtain $\omega = 2$ via  a triple product property construction using three subgroups inside $\GL(n,q)$ for varying $q$ and any fixed $n$, or even inside products of such groups.

Our first barrier result rules out obtaining exponent~$2$ from finite groups of Lie type, but still leaves open the possibility that such groups could serve as building blocks in efficient algorithms for matrix multiplication. For example, the \emph{direct product} of such groups
escapes the barrier entirely, since the second-smallest dimension of an irreducible representation of a direct product equals the second-smallest dimension among the irreducible representations of the factors. Similarly, our normalizer barrier suggests
that constructions should aim to use subgroups that are \emph{self-normalizing}. We therefore view our barriers as giving us useful information about what a possible construction
using finite groups of Lie type must look like, if it is to give $\omega =  2$.

In the second part of the paper, we give constructions that naturally
use a direct product in a critical way (\Cref{thm:slnpow}), and constructions that use self-normalizing
subgroups (\Cref{thm:3conj}). It is important to note that we know of \emph{no} constructions in
\emph{finite} groups of Lie type that even meet a certain ``packing bound'' (\Cref{def:packingbound}), a prerequisite for obtaining $\omega = 2$. This
remains an important challenge. In lieu of such constructions, we direct
our efforts at obtaining constructions in \emph{continuous} Lie groups, which seems easier and
mathematically cleaner to work in, and where we can ask direct analogues
of the main questions. Here we have constructions that meet the packing bound. Moreover, we give examples
that use direct products and self-normalizing subgroups, desiderata by the barrier results.

Our constructions do not achieve the Lie analogue of beating exponent $3$, and we
suggest improving them and finding other examples as key challenges highlighted by this work.

\subsection{Outline}

In the next section we review the group-theoretic approach to bounding $\omega$. In \Cref{sec:barriers} we give our barriers. We then introduce the Lie exponent in \Cref{sec:lieexp} and give our constructions. We conclude with some questions in \Cref{sec:quest}.

\section{Background}

In this section we review the approach of \cite{cu2003} to bounding $\omega$. For a finite group $G$, let $\irr(G)$ denote the set of equivalence classes of irreducible complex representations of $G$, and for $X \subseteq G$, let $Q(X) = \{x x'^{-1}: x,x' \in X\}$ be the quotient set of $X$ and $X^{-1} = \{x^{-1} : x \in X\}$ be its collection of inverses. 

\begin{definition}\cite[Definition 2.1]{cu2003}\label{def:tpp}
We say that subsets $S$, $T$, and $U$ of a group $G$ satisfy the \emph{triple product property} if for all $s \in Q(S)$, $t \in Q(T)$, and $u \in Q(U)$,
\[stu = 1 \implies s=t=u=1.\]
\end{definition}

The simplest case is when the three subsets of $G$ are subgroups, in which case we do not need to take their quotient sets since $Q(H)=H$ for every subgroup $H$. However, the known constructions that achieve nontrivial upper bounds for $\omega$ do not use subgroups \cite{cohn2005group}.

Given three subsets of $G$ that satisfy the triple product property, \cite{cu2003} shows how to reduce a matrix multiplication problem to convolution of functions on $G$ (in other words, multiplication in the group algebra $\C[G]$). This reduction yields the following inequality for the exponent of matrix multiplication $\omega$:

\begin{theorem} \textup{\cite[Theorem 4.1]{cu2003}} \label{thm:omegabound}
If $S$, $T$, and $U$ satisfy the triple product property in $G$, then
\[(|S|\,|T|\,|U|)^{\omega/3} \le  \sum_i d_i^\omega,\]
where $d_i \in \N$ are the dimensions of the irreducible representations of $G$. 
\end{theorem}

If $|S|$, $|T|$, and $|U|$ are large and the dimensions $d_i$ are not too large, then this inequality yields an upper bound for $\omega$. For example, if
\[|S|\,|T|\,|U| >  \sum_i d_i^3,\]
then it proves that $\omega<3$.
Using \Cref{thm:omegabound}, it was shown in \cite{cohn2005group} that $\omega < 2.41$. In fact, this framework is powerful enough to capture the Coppersmith-Winograd family of algorithms, including the latest record of $\omega < 2.37286$ from \cite{alman2021refined}.

In the notation of \Cref{thm:omegabound}, we always have $\sum_i d_i^\omega \ge \sum_i d_i^2 = |G|$. Thus, the inequality in the theorem never implies a better bound on $\omega$ than $(|S| \, |T| \, |U|)^{\omega/3} \le |G|$ would. This inequality is always consistent with $\omega=2$, because the multiplication maps from $S \times T$, $T \times U$, and $S \times U$ to $G$ must be injective to satisfy the triple product property, and therefore $|S| \, |T| \, |U| \le |G|^{3/2}$. This gives us a necessary condition for proving that $\omega=2$ via \Cref{thm:omegabound}:

\begin{definition}\label{def:packingbound}
We say a sequence $G_1,G_2,\dots$ of finite groups \emph{meets the packing bound} if there exist subsets $S_i$, $T_i$, and $U_i$ of $G_i$ satisfying the triple product property such that
\[
|S_i| \, |T_i| \, |U_i| = |G_i|^{3/2-o(1)}
\]
as $i \to \infty$.
\end{definition}

Several families of groups meeting the packing bound were constructed in \cite{cu2003}. In the terminology of \cite{cu2003}, meeting the packing bound means the pseudo-exponent of $G_i$ converges to~$2$ as $i \to \infty$. A simple argument \cite[Lemma~3.1]{cu2003} shows that this can happen only if $|G_i| \to \infty$. 

Meeting the packing bound is a necessary condition for \Cref{thm:omegabound} to yield $\omega=2$: if a family of groups contains no sequence meeting the packing bound, then there is a constant $\varepsilon>0$ such that no group in the family can prove an upper bound on $\omega$ better than $2+\varepsilon$ via \Cref{thm:omegabound}. However, meeting the packing bound is not in itself sufficient to achieve $\omega=2$, or even $\omega<3$. 

\section{Barriers for matrix groups}\label{sec:barriers}

In this section we explain the two barriers mentioned in the introduction. Each of them is based on an idea that is particularly relevant for matrix groups, although we formulate the bounds in greater generality.

\subsection{A representation-theoretic barrier}

We begin by proving our representation-theoretic barrier, which we then apply to groups of Lie type. Our proof of \Cref{thm:gowerstrick} follows the Fourier-analytic proof of Gowers' theorem on mixing in quasirandom groups (see, for example, \cite[Lemma 2.2]{breuillard2014brief}). Our barrier is a function of the second-smallest dimension of an irreducible representation of $G$. Because we use this parameter frequently, we introduce notation for it:

\begin{definition}
For a finite nonabelian group $G$, let $n(G) \coloneqq \min_{\pi \in \irr(G) : \ \dim \pi > 1} \dim \pi$ be the smallest dimension of an irreducible representation of $G$ of dimension greater than $1$.
\end{definition}

\begin{theorem}\label{thm:gowerstrick}
If subsets $S$, $T$, and $U$ satisfy the triple product property in a finite nonabelian group $G$, then \[|S|\,|T|\,|U| \le \frac{|G|^{3/2}}{n(G)^{1/2}} + |G|.\] 
\end{theorem}

\begin{proof}
Let $1_X$ denote the indicator function of a subset $X \subseteq G$. For brevity, given $X \subseteq G$ and $\pi \in \irr(G)$,  we will write $\pi(X) \coloneqq \sum_{x \in X} \pi(x)$ and $d_\pi \coloneqq \dim \pi$.

Suppose that $S,T,U$ satisfy the triple product property. Equivalently, the value at the identity in the 6-fold convolution $1_S*1_{S^{-1}}*1_T *1_{T^{-1}}*1_U *1_{U^{-1}}$ equals $|S|\,|T|\,|U|$. The Fourier inversion formula says that a function $f \colon G \to \C$ can be reconstructed using the inner product $\langle X,Y \rangle = \Tr(X \overline{Y}{}^\top)$ (where $\overline{Y}$ denotes the entry-wise complex conjugate)
as
\[
f(g) = \frac{1}{|G|} \sum_{\pi \in \irr(G)} d_\pi \left\langle \sum_{h \in G} f(h) \pi(h), \pi(g)\right\rangle.
\]
Applying this formula to $f = 1_S*1_{S^{-1}}*1_T *1_{T^{-1}}*1_U *1_{U^{-1}}$ and $g=1$ yields
\[
|G|\,|S|\,|T|\,|U| = \sum_{\pi \in \irr(G)} d_\pi \Tr(\pi(S)\pi(S^{-1}) \pi(T)\pi(T^{-1})\pi(U)\pi(U^{-1})).
\]
When $d_\pi=1$,
\[
\pi(S)\pi(S^{-1}) = \sum_{s \in S} \pi(s) \sum_{s \in S} \overline{\pi(s)} = \left|\pi(S)\right|^2,
\]
which is a nonnegative real number, and $\pi(S) \pi(S^{-1}) = |S|^2$ if $\pi$ is the trivial representation.
Thus,
\begin{align*}
|G|\,|S|\,|T|\,|U| &\ge  (|S|\,|T|\,|U|)^2 + \sum_{\pi  : \ d_\pi > 1}  d_\pi \Tr(\pi(S)\pi(S^{-1}) \pi(T)\pi(T^{-1})\pi(U)\pi(U^{-1})  )\\
&= (|S|\,|T|\,|U|)^2 + \sum_{\pi : \ d_\pi > 1}  d_\pi \Tr(\pi(S^{-1}) \pi(T)\pi(T^{-1})\pi(U)\pi(U^{-1})\pi(S)). \end{align*}
By the Cauchy--Schwarz inequality,
\[|G|\,|S|\,|T|\,|U| \ge (|S|\,|T|\,|U|)^2 - \sum_{\pi : \ d_\pi > 1}  d_\pi \|\pi(S^{-1}T)\| \cdot  \|\pi(T^{-1}U) \| \cdot  \| \pi(U^{-1}S) \|,\]
where $\|\cdot\|$ denotes the Frobenius norm of a matrix (i.e., $\|M \|^2 = \Tr(M \overline{M}{}^\top)$).

Fourier inversion implies a nonabelian version of Parseval's identity, which states that for any function $f \colon G \to \mathbb{C}$, 
\[
 \sum_{g \in G} |f(g)|^2= \frac{1}{|G|}\sum_{\pi \in \irr(G)} d_\pi \left\|\sum_{g\in G}f(g)\pi(g)\right\|^2.
\]
Applying this formula with  $f = 1_{S^{-1}} * 1_T$, which is equal to the indicator function of $S^{-1}T$ by the triple product property, we obtain 
\[|S|\,|T|\,|G| = \sum_{\pi \in \irr(G)} d_\pi \|\pi(S^{-1}T)\|^2,\]
and thus for each $\pi \in \irr(G)$ with $d_\pi > 1$,
\[\|\pi(S^{-1} T)\| \le \sqrt{|S|\,|T|\,|G|/n(G)}.\]
Using this bound and the Cauchy--Schwarz inequality, we find that
\begin{align*}
|G|\,|S|\,|T|\,|U|  &\ge (|S|\,|T|\,|U|)^2 - \sqrt{|S|\,|T|\,|G|/n(G)} \sum_{\pi : \ d_\pi > 1}  d_\pi \|\pi(T^{-1}U)\| \cdot  \|\pi(U^{-1}S) \|\\
&\ge (|S|\,|T|\,|U|)^2 - \sqrt{|S|\,|T|\,|G|/n(G)} \sqrt{\sum_{\pi : \ d_\pi > 1}  d_\pi \|\pi(T^{-1}U)\|^2}  \sqrt{ \sum_{\pi : \ d_\pi > 1}  d_\pi\|\pi(U^{-1}S) \|^2 },
\end{align*}
and so by Parseval's identity,
\begin{align*}
|G|\,|S|\,|T|\,|U| &\ge (|S|\,|T|\,|U|)^2 - \sqrt{|S|\,|T|\,|G|/n(G)} \sqrt{|G|\,|T|\,|U|}  \sqrt{|G|\,|S|\,|U|}\\
 &= (|S|\,|T|\,|U|)^2 - |S|\,|T|\,|U|\,|G|^{3/2}/n(G)^{1/2}.
\end{align*}
We conclude that
\[|S|\,|T|\,|U| \le \frac{|G|^{3/2}}{n(G)^{1/2}} + |G|,\]
as desired.
\end{proof}

We immediately obtain the following corollary:

\begin{corollary} \label{cor:bigreps}
No sequence $G_1,G_2,\dots$ of finite groups satisfying $n(G_i) \ge \Omega(|G_i|^\delta)$ with $\delta>0$ can meet the packing bound.
\end{corollary}

\begin{corollary}\label{cor:liebarrier}
There exists a constant $\varepsilon>0$ such that no
triple product property construction in a group of Lie type can yield an upper bound on $\omega$ better than $2+\varepsilon$.
\end{corollary}

This corollary is more subtle than the previous one, since it does not simply amount to a failure to meeting the packing bound.

\begin{proof}
First, we deal with the case of groups of Lie type of bounded rank. Such groups $G$ satisfy $n(G) \ge \Omega(|G|^\delta)$ for some constant $\delta>0$, as one can check from the bounds given in \cite{landazuri1974minimal}, and this condition suffices by \Cref{cor:bigreps}.

Now let $G$ be a group of Lie type of rank $r$ and dimension $d$ over $\F_q$. Then $|G| = \Theta(q^d)$ (see, for example, \cite[Table~24.1]{MalleTesterman}), and the lower bound $n(G) \ge \Omega(q^r)$ holds by \cite{landazuri1974minimal}. Hence by \Cref{thm:gowerstrick},
\[|S|\,|T|\,|U| \le |G|^{3/2}/\sqrt{n(G)} + |G| = O(q^{3d/2 - r/2}).\]

By \cite[Theorem~1.1]{fulman2012bounds}, there are $O(q^r)$ conjugacy classes in $G$. Let $d_1,\dots,d_m$ be the dimensions of the irreducible representations of $G$, where $m$ is the number of conjugacy classes of $G$.
We have $\sum_i d_i^2 = |G| = \Theta(q^d)$, and
\[
\frac{\sum_i d_i^\omega}{m} \ge \left(\frac{\sum_i d_i^2}{m}\right)^{\omega/2}
\]
since $x \mapsto x^{\omega/2}$ is a convex function. Hence $\sum_i d_i^\omega \ge \Omega(q^{r+\omega(d-r)/2})$, and therefore \Cref{thm:omegabound} cannot yield an upper bound on $\omega$ better than
\[\omega \le 3 \left ( \frac{r+ \log_q C}{r} \right )\]
for some absolute constant $C>0$. If $r$ is large enough, then this bound cannot approach $2$, and the case of bounded $r$ was dealt with above.
\end{proof}

Note that this corollary holds not just for groups of Lie type, but also for simple groups that are quotients of groups of Lie type by their centers. In particular, the second part of argument depends on only two bounds, namely a lower bound for $n(G)$ and an upper bound for the number of conjugacy classes in $G$. Both of these bounds are preserved by taking the quotient: a quotient group always has at most as many conjugacy classes as the original group, and every irreducible representation of the quotient yields an irreducible representation of the original group.

Another consequence of \Cref{thm:gowerstrick} is a slightly sharper estimate for how close $|S| \, |T| \, |U|$ can come to $|G|^{3/2}$ when $S$, $T$, and $U$ satisfy the triple product property in $G$. It follows from \cite[Lemma~3.1]{cu2003} that $|S| \, |T| \, |U| < |G|^{3/2}$, but this inequality  does not rule out the possibility that $|S|$, $|T|$, and $|U|$ might be a large as $\lfloor |G|^{1/2} - 1 \rfloor$. The following corollary shows that this cannot happen when $|G|$ is sufficiently large.

\begin{corollary}
If subsets $S$, $T$, and $U$ satisfy the triple product property in a finite group $G$, then $|S|\,|T|\,|U|\le |G|^{3/2}/\sqrt{2}+|G|$.
\end{corollary}

\begin{proof}
If $G$ is abelian, then $|S|\,|T|\,|U| \le |G|$ by \cite[Lemma~3.1]{cu2003}. Othewise $n(G) \ge 2$ and the conclusion follows from \Cref{thm:gowerstrick}.
\end{proof}

\subsection{A barrier for subgroups that are not self-normalizing}

In contrast to the previous barrier, which follows from properties of the containing group, we now give a barrier in terms of the three subsets used in a triple product property construction. It will apply only to the case of three subgroups, as opposed to arbitrary subsets.

For $X \subseteq G$, let $N(X) = \{g \in G : gXg^{-1} = X\}$
denote the normalizer of $X$ in $G$, and let $Z(G) = \{g \in G : gh = hg \text{ for all } h \in G\}$
denote the center of $G$.

\begin{theorem}\label{thm:normbarrier}
Suppose that subgroups $H_1$, $H_2$, and $H_3$ satisfy the triple product property in a finite group $G$, and let $s_i = |N(H_i)|/|H_i|$. Then
\[|H_1|\,|H_2|\,|H_3| \le \frac{|G|^{3/2}}{(s_1s_2s_3)^{1/4}}.\]
\end{theorem}

\begin{proof}
The main observation in this proof is that $|H_1|\,|N(H_1) \cap H_2|\,|H_3| \le |G|$ (and the analogous inequality for any permutation of $H_1$, $H_2$, and $H_3$). To prove this inequality, we will show that the map
\[(h_1,h_2,h_3) \mapsto h_1 h_2 h_3\]
is injective on $H_1 \times (N(H_1) \cap H_2) \times H_3$. If not, then there exist $(h_1, h_2, h_3) \ne (h_1', h_2', h_3')$ for which \[h_1h_2h_3 = h_1'h_2'h_3',\]
which implies that
\[h_2'^{-1}h_1'^{-1}h_1h_2h_3h_3'^{-1} = 1.\]
However, $h_2'^{-1}(h_1'^{-1}h_1)h_2'$ is another element $h_1'' \in H_1$ (not equal to $1$ if $h_1' \ne h_1$), since $h_2'$ is in the normalizer of $H_1$. We thus have $h_1''(h_2'^{-1}h_2)(h_3h_3'^{-1}) =1$ with not all three factors equal to 1, which contradicts the triple product property for $H_1$, $H_2$, and $H_3$.

Now this inequality implies that
\[|G| \ge |H_1| \,|N(H_1) \cap H_2| \,|H_3| =|H_1| \frac{|N(H_1)| \, |H_2|}{|N(H_1)H_2|} |H_3| \ge |H_1| \frac{|N(H_1)| \, |H_2|}{|G|} |H_3|.\]
The inequality in the theorem statement follows by repeating this argument with $H_2,H_3$ and then $H_3, H_1$ in place of $H_1$ and $H_2$ and taking the product.
\end{proof}

\begin{remark}
The same proof in fact works for subsets $S,T,U$ satisfying the triple product property, not just subgroups, and leads to the conclusion that
\[
|S|\,|T|\,|U| \leq \left(\frac{|G|^3}{|N(Q(S)) \cap T|\, |N(Q(T)) \cap U|\, |N(Q(U)) \cap S|}\right)^{1/2}.
\]
\end{remark}

The following corollary shows that triple product property constructions using subgroups of groups $G$ satisfying $|Z(G)| = \Omega(|G|^\delta)$ with $\delta>0$ cannot meet the packing bound. For example, this shows that triples of subgroups in $\GL(n,q)$ with fixed $n$ cannot meet the packing bound.\footnote{More generally, arbitrary subsets cannot meet the packing bound, because intersecting random translates of the subsets with $\SL(n,q)$ would give subsets of $\SL(n,q)$ meeting the packing bound in expectation, and we have seen that this is impossible since $\SL(n,q)$ a group of Lie type of bounded rank when $n$ is fixed.}

\begin{corollary}\label{cor:centerbarrier}
If subgroups $H_1$, $H_2$, and $H_3$ satisfy the triple product property in a finite group $G$, then
\[|H_1|\,|H_2|\,|H_3| \le \frac{|G|^{3/2}}{|Z(G)|^{1/2}}.\]
\end{corollary}

\begin{proof}
Because $H_1 \cap Z(G)$, $H_2 \cap Z(G)$, and $H_3 \cap Z(G)$ satisfy the triple product property
in the abelian group $Z(G)$,
\[
|Z(G)| \ge |H_1 \cap Z(G)| \,|H_2 \cap Z(G)|\, |H_3 \cap Z(G)|
\]
by \cite[Lemma~3.1]{cu2003}. Combining this inequality with $Z(G) \subseteq N(H_i)$ shows that
\begin{align*}
|Z(G)| &\ge |H_1 \cap Z(G)|\, |H_2 \cap Z(G)| \,|H_3 \cap Z(G)|\\
&= |H_1|\,|H_2|\,|H_3|\,|Z(G)|^3/(|H_1 Z(G)|\,|H_2 Z(G)|\,|H_3 Z(G)|)\\
&\ge |H_1|\,|H_2|\,|H_3|\,|Z(G)|^3/(|H_1 N(H_1)|\,|H_2 N(H_2)|\,|H_3 N(H_3)|)\\
&=|H_1|\,|H_2|\,|H_3|\,|Z(G)|^3/(|N(H_1)|\,|N(H_2)|\,|N(H_3)|),
\end{align*}
and therefore $|N(H_1)|\,|N(H_2)|\,|N(H_3)|/(|H_1|\,|H_2|\,|H_3| ) \ge |Z(G)|^2$. The conclusion now follows by \Cref{thm:normbarrier}.
\end{proof}

\section{Constructions in Lie groups}\label{sec:lieexp}

In this section, we study triple product property constructions in Lie groups (i.e., groups that are also smooth manifolds). All Lie groups will be assumed to be positive-dimensional. We define the Lie exponent of a Lie group $G$ in terms of the rank $r(G)$, which we take to be the real dimension of a Cartan subalgebra of the Lie algebra.\footnote{Note that this differs from the usual convention for complex Lie groups of using the complex dimension. For example, this is why \Cref{table:dimensions} shows that $r(\GL(n,\C))=2n$.}

\begin{definition}\label{def:lieexponent}
The \emph{Lie exponent} $\omega(G)$ of a Lie group $G$ of rank $r(G)$ is the infimum of the quantity
\[\frac{r(G)}{(\dim M_1 + \dim M_2 + \dim M_3)/3 - (\dim G -r(G))/2}\]
over all submanifolds $M_1$, $M_2$, and $M_3$ of $G$ satisfying the triple product property and $(\dim M_1 + \dim M_2 + \dim M_3)/3 > (\dim G -r(G))/2$.
(Recall that the infimum of the empty set is $+\infty$.)
The Lie exponent of a family of groups is the infimum of $\omega(G)$ over $G$ in the family.
\end{definition}

We primarily have in mind semisimple Lie groups, or more generally reductive groups, and it is unclear how relevant the Lie exponent is for other groups. Note that if $G$ is abelian, then $r(G) = \dim_{\mathbb{R}} G$.

\Cref{def:lieexponent} is motivated by the following analogy with the finite field setting. The finite groups of Lie type fall into families of Chevalley groups defined over $\F_q$ as $q$ varies, with the families corresponding to the classification of simple Lie groups (as well as some complications such as twisting). For example, $\SL(n,q)$ is analogous to $\SL(n,\R)$ or $\SL(n,\C)$. Suppose we have triple product property constructions with subsets of sizes $q^{m_1+o(1)}$, $q^{m_2+o(1)}$, and $q^{m_3+o(1)}$ in such a family of simple groups $G_q$ as $q \to \infty$. This is a finite analogue of having submanifolds of dimensions $m_1$, $m_2$, and $m_3$. It follows from \cite[Theorem~1.3]{largestirrep} that the largest irreducible representation of $G_q$ has dimension $q^{(d-r)/2+o(1)}$ as $q \to \infty$, where $d$ and $r$ are the dimension and rank of the corresponding Lie group.\footnote{To deduce this result from \cite[Theorem~1.3]{largestirrep}, note that the Steinberg representation has dimension $q^{(d-r)/2}$. See also \cite[Theorems~5.1--5.3]{largestirrep} for some classical groups that are not quite simple.} 
If the dimensions of the irreducible representations of $G_q$ are $d_1,\dots,d_k$, then
by \Cref{thm:omegabound},
\begin{align*}
q^{(m_1+m_2+m_3 + o(1))\omega/3} &\le \sum_i d_i^\omega\\ 
& \le \sum_i d_i^2 \max_j d_j^{\omega-2} \\
& = |G| \max_j d_j^{\omega-2}\\
& = q^{d + (\omega-2)(d-r)/2 + o(1)},
\end{align*}
and taking the limit as $q \to \infty$ shows that
\[
\omega \le \frac{r}{(m_1+m_2+m_3)/3 - (d-r)/2}
\]
if the denominator is positive.
In other words, \Cref{def:lieexponent} is exactly the bound on $\omega$ one would get if the construction had an analogue in the corresponding finite groups of Lie type. We note that we know of no general reason why such an analogue should exist; indeed, we do not know of any finite analogues of the Lie group constructions given later in this section. We pose the following question:

\begin{question}
Is it true that for every Lie group $G$, the exponent of matrix multiplication is at most $\omega(G)$?
\end{question}

By \Cref{prop:lieexplower} below, the answer must be yes if $\omega=2$. A direct proof would be of considerable interest. Even without such a proof, we view $\omega(G)$ as a model for what groups can do in the continuous setting, which allows for geometric constructions that may not work over finite fields, but that we might hope give inspiration for related constructions in the finite setting.

\begin{proposition}\label{prop:lieexplower}
Every Lie group $G$ has $\omega(G) > 2$.
\end{proposition}

\begin{proof}
If $M_1$, $M_2$, and $M_3$ satisfy the triple product property in $G$, then the map $(m_1,m_2) \mapsto m_1^{-1}m_2$ from $M_1 \times M_2$ to $G$ is injective, and so $\dim M_1 + \dim M_2 \le \dim G$. Similarly, $\dim M_2 + \dim M_3 \le \dim G$ and $\dim M_1 + \dim M_3 \le \dim G$. Averaging these inequalities shows that $(\dim M_1 + \dim M_2 + \dim M_3)/3 \le (\dim G)/2$, and therefore the bound we obtain for $\omega(G)$ is at least 
\[\frac{r(G)}{(\dim G)/2 - (\dim G -r(G))/2} = 2.\]
Equality could hold only if $\dim M_i = (\dim G)/2$ for all $i$. In that case, every continuous, injective function from $M_1 \times M_2$ to $G$ must be open by invariance of domain, and therefore has an open image. In particular, for $m_1' \in M_1$ and $m_2' \in M_2$ consider the map sending $(m_1,m_2) \in M_1 \times M_2$ to $m_1' m_1^{-1} m_2 m_2'^{-1}$. Its image contains a neighborhood of $1$, but by the triple product property it intersects the quotient set $Q(M_3) = M_3 M_3^{-1}$ only at $1$, which is impossible since $Q(M_3)$ contains a submanifold $M_3 m_3^{-1}$ (for fixed $m_3 \in M_3$) of dimension $(\dim G)/2$ that contains $1$.
\end{proof}

Note that in the notation of \Cref{def:lieexponent}, if $\dim M_1 + \dim M_2 + \dim M_3 \le \dim G$, then they cannot prove any better upper bound for $\omega(G)$ than $3$. This conclusion is immediate if $r(G) = \dim G$, and one can check that it follows from $r(G) \le \dim G$.
In fact, the best upper bound we know on the Lie exponent is $3$, which holds for abelian groups:

\begin{proposition}\label{prop:triviallieexp}
If $G$ is abelian, then $\omega(G) = 3$.
\end{proposition}

\begin{proof}
Let $H_1 = G$ and $H_2 = H_3 = \{1\}$. Then $H_1$, $H_2$, and $H_3$ satisfy the triple product property in $G$. Since $r(G) = \dim G$ and $\dim H_1 + \dim H_2 + \dim H_3 = \dim G$, it follows that $\omega(G) \le 3$.

For the other direction, note that if $G$ is abelian, then the product map $M_1 \times M_2 \times M_3 \to G$ must be injective or else the triple product property fails. If the map is injective, then $\dim M_1 + \dim M_2 + \dim M_3 \le \dim G$ and hence $\omega(G) \ge 3$.
\end{proof}

There is also a Lie analogue of the packing bound from \Cref{def:packingbound}.

\begin{definition}\label{def:liepackingbound}
We say a sequence $G_1,G_2,\dots$ of Lie groups \emph{meets the packing bound} if there exist submanifolds $M_{1,i}$, $M_{2,i}$, and $M_{3,i}$ of $G_i$ satisfying the triple product property such that
\[\lim_{i \to \infty} \frac{\dim G_i}{(\dim M_{1,i} + \dim M_{2,i} + \dim M_{2,i})/3} = 2.\]
\end{definition}

\begin{proposition}
If Lie groups $G_1,G_2,\dots$ have $\lim_{i \to \infty} \omega(G_i)=2$, then they 
achieve the packing bound.
\end{proposition}

\begin{proof}
It suffices to show that for $M_1$, $M_2$, and $M_3$ satisfying the triple product property in a group $G$ with $(\dim M_1 + \dim M_2 + \dim M_3)/3 > (\dim G -r(G))/2$,
\[\frac{r(G)}{(\dim M_1 + \dim M_2 + \dim M_3)/3 - (\dim G -r(G))/2} \ge \frac{\dim G }{(\dim M_{1} + \dim M_{2} + \dim M_{3} )/3}.\]
This assertion follows from the inequality $(\dim M_{1} + \dim M_{2} + \dim M_{3})/3 \le (\dim G)/2$ used in the proof of \Cref{prop:lieexplower}.
\end{proof}

It was shown in \cite[Theorem~6.1]{cu2003} that the Lie groups $\SL(n,\R)$ meet the packing bound, by taking $M_1$, $M_2$, and $M_3$ to be the groups of upper unitriangular, lower unitriangular, and orthogonal matrices. In this construction, the group is an algebraic group over $\R$, as are the subgroups $M_i$. In particular, they are all linear algebraic groups over $\R$, i.e., subgroups of $\GL(n,\R)$ defined by polynomial equations. Algebraic groups are a little more general than linear algebraic groups; they are to algebraic varieties as Lie groups are to manifolds.

However, algebraic varieties over $\C$ (or any algebraically closed field) cannot help:

\begin{theorem} \label{thm:alggroup}
Let $G$ be an algebraic group over an algebraically closed field, and let $V_1$, $V_2$, and $V_3$ be subvarieties of $G$ that satisfy the triple product property. Then
\[
\dim V_1 + \dim V_2 + \dim V_3 \le \dim G.
\]
\end{theorem}

As a consequence, subvarieties of algebraic groups over $\C$ cannot be used to meet the packing bound or obtain a better Lie exponent than~$3$.

\begin{proof}
Let $v_i'$ be any element of $V_i$, and define $\varphi\colon V_1 \times V_2 \times V_3 \to G$ by
\[
\varphi(v_1,v_2,v_3) = v_1 v_1'^{-1} v_2 v_2'^{-1} v_3 v_3'^{-1}.
\]
By the triple product property, the only solution of $\varphi(v_1,v_2,v_3) = 1$ is $(v_1',v_2',v_3')$, and so the solution set is zero-dimensional since we are working over an algebraically closed field. However, the fiber dimension theorem \cite[Theorem 17.24]{harris2013algebraic} says the dimension of the solution set must be at least $\dim V_1 + \dim V_2 + \dim V_3 - \dim G$, which yields the desired inequality.
\end{proof}

The intuitive difference between $\R$ and $\C$ here is that a variety can have fewer points over $\R$ than one might expect by counting degrees of freedom. For example, the equation $x_1^2 + \dots + x_n^2 = 0$ defines an $(n-1)$-dimensional variety, which has plenty of points over $\C$, but over $\R$ it consists of just a single point. Fields that are not algebraically closed may lead to an anomalously low number of solutions, and we cannot conclude that a variety is zero-dimensional just because it has only one real point. What Theorem~\ref{thm:alggroup} indicates is that to obtain strong examples, we must either use constructions that are not defined by polynomial equations, or choose equations that have fewer solutions over $\R$ than they do over $\C$. (Note that there are many possibilities that are not defined by polynomial equations. For example, constructions that use complex conjugation generally do not define complex subvarieties.)

Theorem~\ref{thm:alggroup} does rule out one superficially attractive possibility, namely obtaining Lie exponent~$2$ via subvarieties of algebraic groups over $\overline{\F}_q$ and then transitioning to finite groups by using finite subfields of $\overline{\F}_q$ with sizes tending to infinity.

We now give several new constructions over $\R$ with parameters that improve upon the previously known construction from \cite{cu2003}. Our constructions make use of the following observation, which relaxes the triple product property for subgroups.

\begin{lemma}\label{lem:htpp}
Suppose that $H_1$, $H_2$, and $H_3$ are Lie subgroups of a Lie group $G$ and $K$ is a compact subgroup of $G$ such that the equation $h_1h_2h_3 = 1$ with $h_i \in H_i$ implies that $h_1,h_2,h_3 \in K$. Then the Lie exponent of $G$ is at most
\[\frac{r(G)}{(\dim H_1 + \dim H_2 + \dim H_3 -2 \dim K)/3 -(\dim G- r(G))/2}.\]
\end{lemma}

We will refer to this situation as the \emph{$K$-triple product property}. More generally, the same holds for submanifolds $M_i$ such that $q_1q_2q_3 = 1$ with $q_i \in Q(M_i)$ implies $q_1,q_2,q_3 \in K$, but we will need it only for subgroups.

\begin{proof}
By the slice theorem \cite[Theorem~I.2.1]{Audin}, there exist submanifolds $H_i'$ of $H_i$ such that $\dim H_i' = \dim H_i - \dim K$ and no two elements $h,h' \in H_i'$ satisfy $h h'^{-1} \in K$ unless $h=h'$. Then the submanifolds $H_1$, $H_2'$, and $H_3'$ of $G$ satisfy the triple property property and yield the asserted bound. (Note that the first submanifold is $H_1$, not $H_1'$.)
\end{proof}

\subsection{Asymptotic Lie exponent 3}

As mentioned above and shown in \cite[Theorem~6.1]{cu2003}, the upper unitriangular, lower unitriangular, and special orthogonal groups satisfy the triple product property in $\SL(n,\R)$. Since these subgroups have dimension $n(n-1)/2$ inside a group of dimension $n^2-1$, the groups $\SL(n,\R)$ meet the packing bound.

However, the rank of $\SL(n,\R)$ is $n-1$, and so this example does not yield any Lie exponent bound (the denominator in \Cref{def:lieexponent} would vanish). In this section we modify the construction to get a bound on the Lie exponent approaching $3$ for powers of $\SL(n,\R)$.

\begin{theorem}\label{thm:slnpow}
For $m>1$ and $n>1$, the Lie exponent of $\SL(n,\R)^m$ is at most 
\[\frac{3m(n-1)}{m(n-1)-n}.\]
\end{theorem}

In the proof we will denote the $i,j$ entry of a matrix $M$ by $M_{i,j}$. To avoid ambiguity we will use superscripts to index sequences of matrices, with parentheses around the superscripts to distinguish them from exponents. Recall also that a unitriangular matrix is a triangular matrix with diagonal entries equal to~$1$.

\begin{proof}
Let $H_1 = \SO(n,\R)^m$, let 
\[
H_2 = \{ (A^{(1)}, \dots, A^{(m)}) \in \SL(n,\R)^m : \text{each $A^{(i)}$ is upper unitriangular}\},
\]
and let
\[
H_3 = \left\{ (B^{(1)}, \dots, B^{(m)}) \in \SL(n,\R)^m : \text{each $B^{(i)}$ is lower triangular and $\prod_{i=1}^m B^{(i)}_{j,j} = 1$ for all $j$}\right\},
\]
These subgroups have dimensions $mn(n-1)/2$, $mn(n-1)/2$, and $m(n(n+1)/2 - 1)-n$, respectively. We claim that they satisfy the $K$-triple product property in $\SL(n.\R)^m$, where 
\[
K = \{ (C^{(1)}, \dots, C^{(m)}) \in \SL(n,\R)^m : \text{each $C^{(i)}$ is diagonal with $\pm 1$ entries}\}.
\]
Since $\dim \SL(n,\R)^m = m(n^2-1)$ and the rank of $\SL(n,\R)^m$ is $m(n-1)$, while $\dim K = 0$, the claimed bound on the Lie exponent will follow by \Cref{lem:htpp}.

Let $M = (M^{(1)}, \dots, M^{(m)}) \in H_1$, $A = (A^{(1)}, \dots, A^{(m)}) \in H_2$, $B = (B^{(1)}, \dots, B^{(m)}) \in H_3$. We will show that if $MA=B$, then for each $i$, the matrix $M^{(i)}$ is diagonal with $\pm 1$ diagonal entries, in which case the same follows for $A^{(i)}$ and $B^{(i)}$. We will prove by induction on $j$ that $M^{(i)} e_j = \pm e_j$ for all $i$, where $e_1,\dots,e_n$ are the standard basis vectors.

Let the $j$th column of $A^{(i)}$ be $A^{(i)}_j$. For any $i$, $A^{(i)}_1 = e_1$, and so $M^{(i)}e_1 = B^{(i)}_1$. Since $M^{(i)}$ is orthogonal, $M^{(i)}e_1$ and thus also $B^{(i)}_1$ must be unit vectors. In particular,  $|B^{(i)}_{1,1}| \le 1$ for all $i$. But since $\prod_{i=1}^m B_{1,1}^{(i)} = 1$ this forces $B^{(i)}_{1,1} = \pm 1$, which proves the claim for $j=1$.

Now suppose $M^{(i)}e_j = \pm e_j$ for all $j < k$ and all $i$. 
For any $i$, $A^{(i)}_k = e_k + \sum_{j < k} A_{j,k}^{(i)}e_j$ and $B^{(i)}_k = B^{(i)}_{k,k} e_k + \sum_{j > k} B_{j,k}^{(i)}e_j$. From the induction hypothesis we deduce that 
\begin{align*}
M^{(i)}e_k
&= M^{(i)} \big(A^{(i)}_k - \sum_{j < k} A_{j,k}^{(i)}e_j\big)  \\
&= B^{(i)}_{k,k} e_k + \sum_{j > k} B_{j,k}^{(i)}e_j - \sum_{j < k} A_{j,k}^{(i)}(\pm e_j). 
\end{align*}
Since $M^{(i)}$ is orthogonal, $M^{(i)}e_k$ is a unit vector. Because $\prod_i B^{(i)}_{k,k} = 1$, it follows as above that $B^{(i)}_{k,k} = \pm 1$ for all $i$. Hence $M^{(i)} e_k = \pm e_k$ for all $i$, which proves the claim.
\end{proof}

\subsection{Conjugates of rotation groups}

\begin{theorem}\label{thm:3conj}
There are three conjugates of $\O(n,\R)$ inside of $\GL(n,\R)$ satisfying the $K$-triple product property, where $K$ is the subgroup of diagonal matrices with $\pm 1$ entries on the diagonal.
\end{theorem}

This construction meets the packing bound. In particular, it evades the normalizer barrier since the normalizer of $\O(n,\R)$ in $\GL(n,\R)$ is $\R^\times \cdot \O(n,\R)$. However, it does not prove a bound for $\omega(\GL(n,\R))$, because each of these subgroups has dimension $n(n-1)/2$, which equals $(\dim \GL(n,\R) - n)/2$, and $r(\GL(n,\R))=n$. Note that the center of $\GL(n,\R)$ plays no role, and we could just as well have stated the theorem for conjugates of $\SO(n,\R)$ inside $\SL(n,\R)$. We use the slightly more general formulation since it is convenient to allow determinant~$-1$ in the proof by induction.

\begin{proof}
Let $G = \GL(n,\R)$ and $H = \O(n,\R)$. To specify the conjugates of $H$, we take
$D_1 = \diag(x_1, \dots, x_n)$ with $x_1 > x_2 > \dots > x_n > 0$ and $D_2 = \diag(y_1, \dots, y_n)$ with $0 < y_1 < y_2 < \dots < y_n$. Then we will show that $H$, $H_1\coloneqq D_1 H D_1^{-1}$, and $H_2\coloneqq D_2 H D_2^{-1}$ satisfy the $K$-triple product property in $G$, where $K$ is the group of diagonal $\pm 1$ matrices.

In particular, we will show that for every $h_1 \in H_1$ and $h_2 \in H_2$, if $h_1^\top h_1 =
h_2^\top h_2$, then $h_1, h_2 \in K$. The conclusion then follows, since if $hh_1h_2^{-1}=I$ with $h \in H$, then $h_1h_2^{-1} \in H$, meaning $(h_1h_2^{-1})^\top(h_1h_2^{-1}) = I$, and hence $h_1^\top h_1 = h_2^\top  h_2$.

Suppose that $h_1 = D_1M_1D_1^{-1}$ and $h_2 = D_2M_2D_2^{-1}$, where
$M_1, M_2 \in H$, and consider
\[h_1^\top h_1 = (D_1^{-1}M_1^\top D_1)(D_1M_1D_1^{-1}).\]
If $(a_1, a_2, \dots, a_n)$ is the first column of $M_1$, then the
first column of $D_1M_1D_1^{-1}$ is \[(a_1,x_2x_1^{-1}a_2,\dots,x_nx_1^{-1}a_n)\]
as is the first row of $D_1^{-1}M_1^\top D_1$,
so the
upper-left entry of their product $h_1^\top h_1$ is \[a_1^2 + (x_2/x_1)^2a_2^2 +
(x_3/x_1)^2a_3^2 + \dots + (x_n/x_1)^2a_n^2.\]
Now, since $a_1^2 + a_2^2 + \dots + a_n^2 = 1$, we can substitute for
$a_1^2$ to obtain
\[1 + ((x_2/x_1)^2 - 1)a_2^2 +
((x_3/x_1)^2 - 1)a_3^2 + \dots + ((x_n/x_1)^2 - 1)a_n^2.\]

Because $x_i>x_1$ for all $i>1$, this quantity is at most~$1$, with equality exactly
when $a_2^2 = a_3^2 = \dots = a_n^2 = 0$.
By an identical argument, if $(b_1, b_2, \dots, b_n)$ is the first column of $M_2 $, then the upper-left entry of $h_2^\top h_2$ is 
\[1 + ((y_2/y_1)^2 - 1)b_2^2 +
((y_3/y_1)^2 - 1)b_3^2 + \dots + ((y_n/y_1)^2 - 1)b_n^2,\]
 which is at least~$1$, with equality exactly
when $b_2^2 = b_3^2 = \dots = b_n^2 = 0$. So if $h_1^\top h_1 =
h_2^\top h_2$, then in particular their upper-left entries are equal, and 
we conclude that $a_1^2 = b_1^2=1$, while $a_i=b_i=0$ for all $i>1$.

Now $h_1$ has the form
\[\left (\begin{array}{c|ccc} \pm 1 & 0 & \dots & 0\\ \hline 0 \\ \vdots &
                                                                  &
                                                                    h_1'
           \\ 0  \end{array} \right ),\]
where $h_1'$ is an element of $D_1' \O(n-1,\R)D_1'^{-1}$ with $D_1' = \diag(x_2, \dots,
x_n)$, and $h_2$ has the form
\[\left (\begin{array}{c|ccc} \pm 1 & 0 & \dots & 0\\ \hline 0 \\ \vdots &
                                                                  &
                                                                    h_2'
           \\ 0  \end{array} \right ),\]
where $h_2'$ is an element of $D_2' \O(n-1,\R) D_2'^{-1}$ with $D_2' = \diag(y_2, \dots
y_n)$. Finally, since $h_1h_2^{-1}h =1$ we find $h$ also has the
same block-diagonal form, and so $h_1'(h_2')^{-1} \in \O(n-1,\R)$,
and then we are done by induction on $n$.
\end{proof}

\begin{remark}
This theorem holds when we replace $G$ with $\GL(n,\C)$ (resp.,~$\GL(n,\Quat)$), $H$ with $\U(n,\C)$ (resp.,~$\Sp(n)$), and $K$ with the group of diagonal matrices with unit complex (resp.,~quaternionic) numbers on the diagonal. This follows from a similar argument, where one replaces transpose with conjugate transpose and uses the positivity of the complex/quaternionic norm. The corresponding dimensions are shown in Table~\ref{table:dimensions}.
\end{remark}

\begin{table}
\caption{Parameters for the real, complex, and quaternionic versions of \Cref{thm:3conj}.}
\label{table:dimensions}
\begin{center}
\begin{tabular}{llll}
\toprule
(skew) field      & $\R$        & $\C$      & $\Quat$       \\
\midrule
$\dim G$ & $n^2$      & $2n^2$   & $4n^2$    \\
$r(G)$ & $n$      & $2n$   & $4n$    \\
$\dim H$ & $n(n-1)/2$ & $n^2$ & $n(2n+1)$ \\
$\dim K$ & $0$        & $n$      & $3n$    \\
Meets packing bound as $n \to \infty$ & yes & yes & yes \\
Lie exponent upper bound & $\infty$ & $6$ & $4$\\
\bottomrule
\end{tabular}
\end{center}
\end{table}

\section{Open problems}\label{sec:quest}

The most important challenge highlighted by this paper is to find a construction proving that the Lie exponent of a family of Lie groups approaches $2$, or to prove that such a construction cannot exist. 
Many questions can be asked along the way, including whether there is a Lie group with Lie exponent less than $3$, and whether the Lie exponent of $\SL(n,\R)$ is even finite.

It follows from \cite{sawin2018bounds} that triple product property constructions inside $\SL(2,q)^m$ cannot give $\omega = 2$ for fixed $q$ and growing $m$, and \Cref{thm:gowerstrick} shows that $\SL(n,q)^m$ cannot give $\omega = 2$ for fixed $m$ and growing $q$. The proofs of these two facts are quite different: one uses the polynomial method, and the other is Fourier analytic. Is there a common generalization of these two facts that would rule out obtaining $\omega = 2$ with $m$ and $q$ both growing?

Together with the fact that abelian groups cannot yield exponent less than 3, \Cref{thm:gowerstrick} implies that the alternating groups are the only simple groups left that could yield $\omega = 2$ via a triple product property construction. The representation-theoretic argument fails in this case, since $A_n$ has an irreducible representation of dimension $n-1$ but $|A_n| = n!/2$. Can an alternate argument rule out these groups?

\section*{Acknowledgments}
This work was supported by an AIM SQuaRE grant. 
J.~B.\ was supported by NSF grant DMS-1855784,   
J.~A.~G.\ was partially supported by NSF CAREER award CISE-2047756, and C.~U.\ was supported by NSF grant CCF-1815607 and a Simons Foundation Investigator award.

\newcommand{\etalchar}[1]{$^{#1}$}
\providecommand{\bysame}{\leavevmode\hbox to3em{\hrulefill}\thinspace}
\providecommand{\MR}{\relax\ifhmode\unskip\space\fi MR }
\providecommand{\MRhref}[2]{%
  \href{http://www.ams.org/mathscinet-getitem?mr=#1}{#2}
}
\providecommand{\href}[2]{#2}

\end{document}